\documentclass[12pt]{article}
\usepackage{epsfig}
\usepackage{amsmath}
\usepackage{amssymb}

\newcommand{\R}{\ensuremath{\mathbb R}}

\newcommand{\Z}{\ensuremath{\mathbb Z}}
\newtheorem{lemma}[equation]{Lemma}
\newtheorem{theorem}[equation]{Theorem}

\newcommand{\qed}{\hspace{\stretch{1}}$\Box$}
\newenvironment{proof}{\vspace{-.25\baselineskip}\noindent\textbf{Proof.}
}{\qed\par\medskip}

\title{Many collinear $k$-tuples with no $k+1$ collinear points}
\author{{J\'ozsef Solymosi}\thanks{Department of Mathematics, University of British Columbia,
Vancouver, Canada, email: solymosi@math.ubc.ca. Supported by NSERC, ERC-AdG.\ 321104, and OTKA NK 104183 grants.} \and {Milo\v{s}
Stojakovi\'{c}}
\thanks{Department of
Mathematics and Informatics, University of Novi Sad, Serbia, email:
milos.stojakovic@dmi.uns.ac.rs. Partly supported by Ministry of Science and Technological
Development, Republic of Serbia, and Provincial Secretariat for Science, Province of
Vojvodina.}}
\begin{document}

\maketitle

\begin{abstract}
For every $k>3$, we give a construction of planar point sets with many collinear $k$-tuples
and no collinear $(k+1)$-tuples. We show that there are $n_0=n_0(k)$ and
$c=c(k)$ such that if $n\geq n_0$, then there exists a set of $n$ points in the plane that does not contain $k+1$ points on a line, but it contains at least $n^{2-\frac{c}{\sqrt{\log n}}}$ collinear $k$-tuples of points. Thus, we significantly improve the previously best known lower bound for the largest number of collinear $k$-tuples in such a set, and get reasonably close to the trivial upper bound $O(n^2)$.
\end{abstract}

\section{Introduction}
In the early 60's Paul Erd\H os asked the following question about point-line incidences in
the plane: \emph{Is it possible that a planar point set contains many collinear four-tuples,
but it contains no five points on a line?} There are constructions for $n$-element
point sets with $n^2/6-O(n)$ collinear triples with no four on a line (see \cite{BGS} or
\cite{FP}). However, no similar construction is known for larger tuples.

Let us formulate Erd\H os' problem more precisely. For a finite set $P$ of points in the plane
and $k\geq 2$, let $t_k(P)$ be the number of lines meeting $P$ in exactly $k$ points, and let
$T_k(P):=\sum_{k'\geq k}t_{k'}(P)$ be the number of lines meeting $P$ in at least $k$ points.
For $r>k$ and $n$, we define
$$
t_k^{(r)}(n):= \max_{|P|=n \atop T_r(P)=0} t_k(P).
$$

In plain words, $t_k^{(r)}(n)$ is the number of lines containing exactly $k$ points from $P$,
maximized over all $n$ point sets $P$ that do not contain $r$ collinear points. Erd\H os conjectured
that $t_k^{(r)}(n)=o(n^2)$ for any fixed $r>k>3$ and offered \$100 for a proof or disproof \cite{EP4}
(the conjecture is listed as Conjecture 12 in the problem collection of Brass, Moser, and Pach
\cite{BMP}). In this paper we are concerned about bounding $t_k^{(k+1)}(n)$ from below for $k>3$. To simplify notation, from now on we will use $t_k(n)$ to denote $t_k^{(k+1)}(n)$.

\begin{figure}[htb]
  \centering %
  \epsfig{file=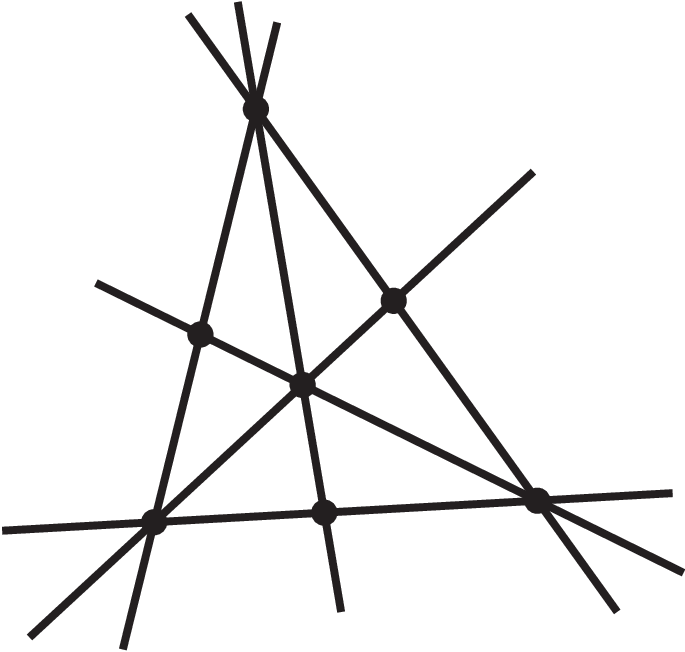, scale=0.5}
  \caption{A construction of a point set showing that $t_3(7)\geq 6$. \label{f:0}}
  \bigskip
\end{figure}

\subsection{Earlier results and our result}
This problem was among Erd\H os' favourite geometric problems, he frequently talked about it
and listed it among the open problems in geometry, see \cite{EP4,E1,E2,E3,EP}. It is not just a
simple puzzle which might be hard to solve, it is related to some deep and difficult problems
in other fields. It seems that the key to attack this question would be to understand the
group structure behind point sets with many collinear triples. A recent result of Green and Tao
-- proving the Motzkin-Dirac conjecture \cite{GT} -- might be  an important development in this direction.

In the present paper, our goal is to give a construction showing that Erd\H os
conjecture, if true, is sharp; for $k>3$, one can not replace the exponent 2 by $2-c$, for
any $c>0$.

The first result was due to K\'arteszi \cite{Ka} who proved that $t_k(n)\geq
c_kn\log{n}$ for all $k>3$. In 1976 Gr\"unbaum \cite{Gr} showed that $t_k(n)\geq
c_kn^{1+1/(k-2)}$. For some 30 years this was the best bound when Ismailescu \cite{Is}, Brass
\cite{Brass}, and Elkies \cite{El} consecutively improved Gr\"unbaum's bound for $k\geq 5$.
However, similarly to Gr\"unbaum's bound, the exponent was going to 1 as $k$ went to infinity.

In what follows we are going to give a construction that substantially improves the lower bound. Namely, we will show the following.

\begin{theorem} \label{t:even}
For any $k\geq 4$ integer, there is a positive integer $n_0$ such that for $n>n_0$ we have
$t_k(n)>n^{2-\frac{c}{\sqrt{\log n}}}$, where $c=2\log(4k+9)$.
\end{theorem}

We note that each of the collinear $k$-tuples
that we count in our construction has an additional property that
the points form a $k$-term arithmetic progression, as the distance
between every two consecutive points is the same in every coordinate.

\subsection{Preliminaries}
For $r>0$ and a positive integer $d$  let $B_d(r)$ denote the closed ball in
$\R^d$ of radius $r$ centred at the origin, and $S_d(r)$ denotes the sphere in $\R^d$ of
radius $r$ centred at the origin.

For any positive integer $d$ the $d$-dimensional unit hypercube (and its translates) are denoted by $H_d$. If the
center of the cube is a point ${\bf x}\in \R^d$ then it has the vertex set ${\bf x}+\{-1/2,1/2\}^d$ and we denote it by
$H_d({\bf x})$.

For a set $S\subseteq \R^d$, let $N(S)$ denote the number of points
from the integer lattice $\Z^d$ that belong to $S$, i.e.,
$N(S):=|\Z^d \cap S|$.

Through the paper the $\log$ notation stands for the base 2 logarithm.


\section{A lower bound on $t_k(n)$}
We will prove bounds for even and odd value of $k$
separately, as the odd case needs a bit more attention. Our proof is
elementary, we use the fact that the volume of a large sphere
approximates well the number of lattice points inside the ball.
There are more advanced techniques to count lattice points on the
surface of a sphere, however we see no way to improve our bound
significantly by applying them.

In our construction, we rely on the fact that a point set in a large dimensional space that satisfies our collinearity conditions can be converted to a planar point set by simply projecting it to a plane along a suitably chosen vector, with all the collinearities preserved. That enables us to perform most of the construction in a space of large dimension, exploiting the properties of such a space.

\subsection{Proof for even $k$}

Let $d$ be a positive integer, and let $r>0$. We will use a quantitative version of the following well known fact
\begin{eqnarray*}
{{N(B_d(r))}\over {V(B_d(r))}}\rightarrow 1 \; \; \; \text{as} \;\: r\rightarrow \infty,
\end{eqnarray*}
where
\begin{eqnarray} \label{vball}
V(B_d(r))=\frac{r^d\pi^{\frac{d}{2}}}{\Gamma\left(\frac{d}{2}+1\right)}, \
\end{eqnarray}
estimating the number of lattice points on a sphere using Gauss' volume argument.

\begin{lemma} \label{l:1}
For $r\geq \sqrt{d}$, we have
$$
V(B_d(r-\sqrt{d}/2))\leq N(B_d(r))\leq V(B_d(r+\sqrt{d}/2)).
$$
\end{lemma}

\begin{proof}
For every lattice point $p\in B_d(r) \cap \Z^d$ we look at the unit cube $H_d(p)$ with center $p$. These cubes all have
disjoint interiors and each of them has diameter $\sqrt{d}$. Moreover, their union $\cup_{p\in B_d(r) \cap \Z^d}
H_d(p)$ is included in $B_d(r+\sqrt{d}/2)$, it contains $B_d(r-\sqrt{d}/2)$ and its volume is equal to $N(B_d(r))$,
which readily implies the statement of the lemma.
\end{proof}

We will also use a bound on the number of points on a sphere.

\begin{lemma} \label{l:2}
There exists a constant $c_0>0$ such that
$$
N(S_d(r))\leq 2^{\frac{c_0\log{r}}{\log\log{r}}}N(B_{d-2}(r)).
$$
\end{lemma}

\begin{proof} The number $N(S_d(r))$ is the number of ways $r^2$ can be written as an ordered sum of $d$ perfect squares. Such
sum can be broken into two sums, the first containing all summands except the last two, and the second containing the
last two summands, so we have
$$
N(S_d(r))=\sum_{s=1}^{r^2}N(S_{d-2}(\sqrt{s}))N(S_2(\sqrt{r^2-s})).
$$

The number of ways a positive integer $n$ can be represented as a sum of two squares is known to be at most 4 times
$d(n)$, the number of divisors of $n$, and there exists a constant $c'>0$ so that $d(n)\leq
2^{\frac{c'\log{n}}{\log\log{n}}}$ (see, e.g.,~\cite[Section 13.10]{Apostol}). Hence, we have
\begin{align*}
N(S_d(r)) &\leq 4 d(r^2)\sum_{s=1}^{r^2}N(S_{d-2}(\sqrt{s})) \\
& \leq 2^{\frac{2c'\log{r}}{\log\log{r}}}N(B_{d-2}(r)).
\end{align*}
\end{proof}

\begin{proof} \textbf{(of Theorem~\ref{t:even} for even $k$)}

We will give a construction of a point set $P$ containing no $k+1$ collinear points, with a high value of $t_k(P)$.

For a positive integer $d$ let us set $r_0=2^{d}$. For each integer point from $B_d(r_0)$, the square of its distance
to the origin is at most $r_0^2$. As the square of that distance is an integer, we can apply pigeonhole principle to
conclude that there exists $r$, with $0< r \leq r_0$, such that the sphere $S_d(r)$ contains at least $1/r_0^2$
fraction of points from $B_d(r_0)$, and together with Lemma~\ref{l:1}, we have
\begin{equation*} \label{n-pts}
N(S_d(r))\geq \frac{1}{r_0^2} N(B_d(r_0)) \geq \frac{1}{r_0^2} V(B_d(r_0-\sqrt{d}/2)).
\end{equation*}

Let us consider the unordered pairs of different points from $\Z^d \cap S_d(r)$. The total number
of such pairs is at least
$$
\binom{N(S_d(r))}{2} \geq \binom{\frac{V(B_d(r_0-\sqrt{d}/2))}{r_0^2}}{2}.
$$
For every $p,q\in \Z^d \cap S_d(r)$ the Euclidean distance $d(p,q)$ between $p$ and $q$ is at most $2r$, and the square
of that distance is an integer. Hence, there are at most $4r^2$ different possible values for $d(p,q)$. Applying
pigeonhole principle again, we get that there are at least
\begin{equation} \label{e:pairs}
\frac{1}{4r^2} \binom{\frac{V(B_d(r_0-\sqrt{d}/2))}{r_0^2}}{2}\geq \frac{V(B_d(r_0-\sqrt{d}/2))^2}{8r_0^6}
\end{equation}
pairs of points from $\Z^d \cap S_d(r)$ that all have the same distance. We denote that
distance by $\ell$.

Let $p_1,q_1\in \Z^d \cap S_d(r)$ with $d(p_1,q_1)=\ell$, and let $s$ be the line going through $p_1$ and $q_1$. We
define $k-2$ points $p_2,\dots,p_{k/2}$, $q_2, \dots, q_{k/2}$ on the line $s$ such that $d(p_i,p_{i+1})=\ell$ and
$d(q_i,q_{i+1})=\ell$, for all $1\leq i<k/2$, and all $k$ points $p_1,\dots,p_{k/2}$, $q_1, \dots, q_{k/2}$ are
different, see Figure~\ref{f:1}.

\begin{figure}[htb]
  \centering %
  \epsfig{file=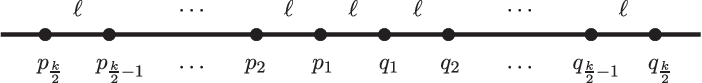, scale=1}
  \caption{Line $s$ with $k$ points, for $k$ even.\label{f:1}}
  \bigskip
\end{figure}

Knowing that $p_1$ and $q_1$ are points from $\Z^d$, the way we defined points \linebreak $p_2,\dots,p_{k/2}, q_2,
\dots, q_{k/2}$ implies that they have to be in $\Z^d$ as well. If we set
\begin{equation} \label{e:ri-even}
r_i:=\sqrt{r^2 + i(i-1)\ell^2}\leq\sqrt{r^2+ (2ir)^2}\leq r(k+1),
\end{equation}
for all $i=1,\dots,k/2$, then the points
$p_i$ and $q_i$ belong to the sphere $S_d(r_i)$, and hence, $p_i, q_i\in \Z^d \cap S_d(r_i)$, for all $i=1,\dots,k/2$,
see Figure~\ref{f:2}.

\begin{figure}[htb]
  \centering %
  \epsfig{file=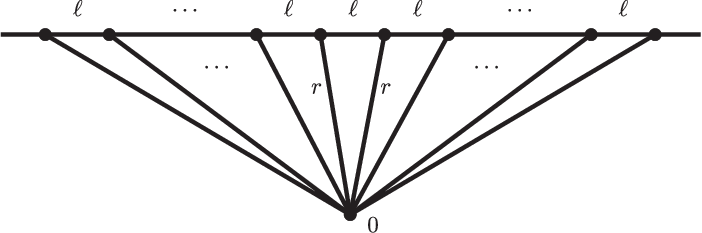, scale=1}
  \caption{The position of the $k$ points related to the origin, for $k$ even.\label{f:2}}
  \bigskip
\end{figure}

We define the point set $P$ to be the set of all integer points on spheres $S_d(r_i)$, for all
$i=1,\dots,k/2$, i.e.,
$$
P:=\Z^d \cap \left( \cup_{i=1}^{k/2} S_d(r_i) \right).
$$
As the point set $P$ is contained in the union of $k/2$ spheres, there are obviously no $k+1$ collinear points in $P$.
On the other hand, every pair of points $p_1,q_1\in \Z^d \cap S_d(r)$ with $d(p_1,q_1)=\ell$ defines one line that
contains $k$ points from $P$.

If we set $n:=|P|$, and if $d$ is large enough, we have
\begin{align}
n & = \sum_{i=1}^{k/2} N(S_d(r_i)) \notag \\
&\leq \sum_{i=1}^{k/2}2^{\frac{c_1\log{r_i}}{\log\log{r_i}}}N(B_{d-2}(r_i))\notag  \\
& \leq k 2^{\frac{c_2 d}{\log{d}}} N\left(B_{d-2}((k+1)2^d)\right) \label{e:third}\\
& \leq k 2^{\frac{c_2 d}{\log{d}}} V\left(B_{d-2}\left(2^d\left(k+1+\sqrt{d} / 2^{d+1}\right)\right)\right)\notag  \\
& \leq k 2^{\frac{c_2 d}{\log{d}}} \frac{2^{d(d-2)} \left(k+1+\sqrt{d} / 2^{d+1} \right)^{d-2} \pi^{(d-2)/2} }{ \frac{\sqrt{\pi d} \left(d/2e\right)^{d/2}}{d/2} } \notag \\
& \leq \frac{2^{d(d-2)} \left(k+2\right)^{d} \pi^{d/2} (2e)^{d/2}}{ d^{d/2} } \notag \\
& \leq 2^{d^2- \frac{d}{2}\log d + \frac{d}{2} \log\left( \frac{(k+2)^2\pi e}{8}\right) } \notag ,
\end{align}
were $c_1,c_2$ are constants depending only on $k$. Here, we used Lemma~\ref{l:1}, Lemma~\ref{l:2}, (\ref{vball}),
(\ref{e:ri-even}), and the standard estimates for the function $\Gamma$.

On the other hand, we get from (\ref{e:pairs}) and (\ref{vball}) that the number of lines containing exactly $k$ points
from $P$ is
\begin{align*}
t_k(P) & \geq \left( \frac{\left( 2^d -\frac{\sqrt{d}}{2}\right)^d \pi^{d/2} (2e)^{d/2}}{\sqrt{\frac{4\pi}{d}} \, d^{d/2}} \right)^2 \frac{1}{2^{6d+3}} \\
& \geq 2^{2d^2- d\log d - d \log \left( \frac{64}{\pi e}\right)} .
\end{align*}


Putting the previous two inequalities together it follows that there
exist a constant $n_0$ depending on $k$, such that for $n>n_0$ we
have
$$
t_k(P) \geq  \frac{n^2}{2^{c\sqrt{\log{n}}}} = n^{2-\frac{c}{\sqrt{\log{n}}}},
$$
where $c=2\log(3k+6)$.

To obtain a point set in two dimensions, we project the $d$ dimensional point set to a two dimensional plane in $\R^d$.
The vector $v$ along which we project should be chosen generically, so that every two points from our point set are
mapped to different points, and every three points that are not collinear are mapped to points that are still not
collinear.
\end{proof}

\subsection{Proof for odd $k$}
\begin{proof} \textbf{(of Theorem~\ref{t:even} for odd $k$)}

We will give a construction of a point set $P$ containing no $k+1$ collinear points, with a high value of $t_k(P)$.

For a positive integer $d$ let us set $r_0=2^{d}$. In the same way as in the proof for even $k$, we can find $r$ with
$0< r \leq r_0$, such that the sphere $S_d(r)$ contains at least a $1/r_0^2$ fraction of the integer points from
$B_d(r_0)$, and hence
$$
N(S_d(r))\geq \frac{1}{r_0^2} N(B_d(r_0)) \geq \frac{1}{r_0^2} V(B_d(r_0-\sqrt{d}/2)).
$$

Now, for every point $p\in \Z^d \cap S_d(r)$ there is a corresponding point $p'$ on the sphere $S_d(2r)$ that belongs
to the half-line from the origin to $p$. It is not hard to see that all coordinates of $p'$ are even integers, so
$p'\in (2\Z)^d \cap S_d(2r)$. Hence, the number of points in $(2\Z)^d \cap S_d(2r)$ is at least $N(S_d(r))$.

We look at unordered pairs of different points from $(2\Z)^d \cap S_d(2r)$. The total number
of such pairs is at least
$$
\binom{N(S_d(r))}{2} \geq \binom{\frac{V(B_d(r_0-\sqrt{d}/2))}{r_0^2}}{2}.
$$
If we just look at such pairs of points that have different first coordinate, we surely have at least half as many
pairs as before. To see that, observe that for every point $p\in(2\Z)^d \cap S_d(2r)$, a point obtained from $p$ by
changing the sign of any number of coordinates of $p$ and/or permuting the coordinates is still in $(2\Z)^d \cap
S_d(2r)$.

For every $p,q\in (2\Z)^d \cap S_d(2r)$ we know that the Euclidean distance $d(p,q)$ between $p$ and $q$ is at most
$4r$, and that the square of that distance is an integer. Hence, there are at most $16r^2$ different possible values
for $d(p,q)$. Applying pigeonhole principle again, we get that there are at least
\begin{equation} \label{e:oddpairs}
\frac{1}{16r^2} \cdot \frac{1}{2}\binom{\frac{V(B_d(r_0-\sqrt{d}/2))}{r_0^2}}{2}\geq \frac{V(B_d(r_0-\sqrt{d}/2))^2}{64
r_0^6}
\end{equation}

pairs of points from $(2\Z)^d \cap S_d(2r)$ with different first coordinate that have the same distance. We denote that
distance by $2\ell$. Note that since both $p$ and $q$ are contained in $(2\Z)^d$, we have that the middle point $m$ of
the segment $pq$ belongs to $\Z^d$, and $d(p,m)=d(q,m)=\ell$.

Let $p_1,q_1\in (2\Z)^d \cap S_d(2r)$ with $d(p_1,q_1)=2\ell$, let $m_0$ be the middle point
of the segment $p_1q_1$, and let $s$ be the line going through $p_1$ and $q_1$. We define
$k-3$ points $p_2,\dots,p_{(k-1)/2}$, $q_2, \dots, q_{(k-1)/2}$ on the line $s$ such that
$d(p_i,p_{i+1})=\ell$ and $d(q_i,q_{i+1})=\ell$, for all $1\leq i<(k-1)/2$, and all $k$ points
$m_0,p_1,\dots,p_{(k-1)/2}$, $q_1, \dots, q_{(k-1)/2}$ are different, see Figure~\ref{f:3}.

\begin{figure}[htb]
  \centering %
  \epsfig{file=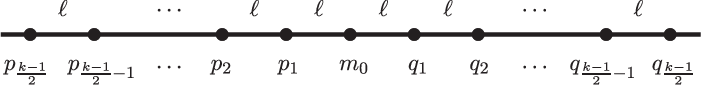, scale=1}
  \caption{Line $s$ with $k$ points, for $k$ odd.\label{f:3}}
  \bigskip
\end{figure}

Knowing that $p_1$ and $q_1$ are points from $(2\Z)^d$, the way we defined points \linebreak
$m_0, p_2,\dots,p_{(k-1)/2}$, $q_2, \dots, q_{(k-1)/2}$ implies that they have to be in
$\Z^d$. If we set
\begin{equation} \label{e:ri-odd}
r_i:=\sqrt{4r^2 + (i+1)(i-1)\ell^2} \leq r(k+1),
\end{equation}
for all $i=0,\dots,(k-1)/2$, the points $p_i$ and $q_i$
belong to the sphere $S_d(r_i)$, and the point $m_0$ belongs to $S_d(r_0)$. Hence, $p_i, q_i\in \Z^d \cap S_d(r_i)$,
for all $i=1,\dots,(k-1)/2$, and $m_0\in\Z^d \cap S_d(r_0)$, see Figure~\ref{f:4}.

\begin{figure}[htb]
  \centering %
  \epsfig{file=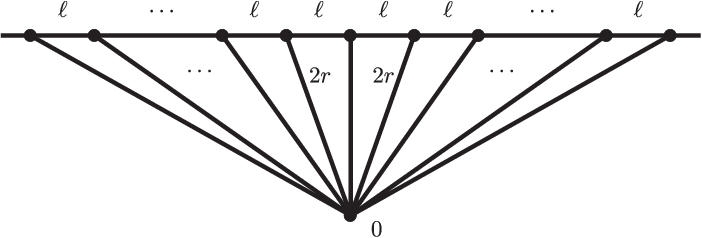, scale=1}
  \caption{The position of the $k$ points related to the origin, for $k$ odd.\label{f:4}}
  \bigskip
\end{figure}

By $\alpha_x$ we denote the hyperplane containing all points in
$\R^d$ with first coordinate equal to $x$. Let $M$ be the multiset
of points $m$ such that there exist points $p,q\in (2\Z)^d \cap
S_d(2r)$ having different first coordinate, with $d(p,q)=2\ell$, and
with $m$ being the middle point of the segment $pq$. In this
multiset, we include the point $m$ once for every such $p$ and $q$.
We know that $M\subseteq \Z^d \cap S_d(r_0)$, and each point from
$\Z^d \cap S_d(r_0)$ is contained in $\alpha_x$ for some integer
$-r_0\leq x \leq r_0$. Hence, by the pigeonhole principle, we get
from (\ref{e:oddpairs}) that there exists $-r_0\leq x_0 \leq r_0$
such that $\alpha_{x_0} \cap M$ contains at least
\begin{equation} \label{e:pairs-odd}
\frac{|M|}{2r_0}\geq \frac{V(B_d(r_0-\sqrt{d}/2))^2}{128 r_0^7}
\end{equation}
points.

We define the point set $P$ to be the set of all integer points on spheres $S_d(r_i)$, for all
$i=1,\dots,(k-3)/2$, all integer points on $S_d(r_{(k-1)/2})$ that do not belong to
$\alpha_{x_0}$, and all integer points on $S_d(r_0)$ that belong to $\alpha_{x_0}$. I.e., we
have
$$
P:=\Z^d \cap \left(\left( \cup_{i=1}^{(k-3)/2} S_d(r_i) \right) \cup \left( S_d(r_{(k-1)/2})
\setminus \alpha_{x_0} \right) \cup \left( S_d(r_0)\cap \alpha_{x_0} \right) \right).
$$

Let us first prove that the point set $P$ does not contain $k+1$ collinear points. As $P$ is contained in the union of
$(k-1)/2$ spheres and a hyperplane, any line that is not contained in that hyperplane cannot contain more than $k$
points from $P$. But the point set $P$ restricted to the hyperplane $\alpha_{x_0}$ belongs to the union of $(k-1)/2$
spheres $S_d(r_i)$, for $i=0,\dots,(k-3)/2$, so we can also conclude that there are no $k+1$ collinear points in $P\cap
\alpha_{x_0}$.

Let $n:=|P|$. Obviously, $P\subseteq \cup_{i=0}^{(k-1)/2} S_d(r_i)$, so we can estimate the value of $n$ similarly as in the even case. We have
\begin{align*}
n & \leq \sum_{i=0}^{\frac{k-1}{2}} N(S_d(r_i)) \\
&\leq \sum_{i=0}^{\frac{k-1}{2}} 2^{\frac{c_1\log{r_i}}{\log\log{r_i}}}N(B_{d-2}(r_i)) \\
& \leq k2^{\frac{c_2 d}{\log{d}}}N\left(B_{d-2}((k+1)2^d)\right) \\
 & \leq 2^{d^2- \frac{d}{2}\log d + \frac{d}{2} \log\left( \frac{(k+2)^2\pi e}{8}\right) },
\end{align*}
were $c_1,c_2$ are constants depending only on $k$. The last estimate was done the same way as in the case where $k$ is
even, as the third line of the calculation is exactly the same as the one obtained in~(\ref{e:third}). 

On the other hand, every pair of points $p_1,q_1\in \Z^d \cap S_d(r)$ with different first coordinate, with
$d(p_1,q_1)=2\ell$, and with the middle point that belongs to $\alpha_{x_0} \cap M$, defines one line that contains $k$
points from $P$. Note that such line cannot belong to $\alpha_{x_0}$, as the first coordinates of $p_1$ and $q_1$
cannot be $x_0$ simultaneously.

Hence, we get from (\ref{e:pairs-odd}) and (\ref{vball}) that the
number of lines containing exactly $k$ points from $P$ is
\begin{align*}
t_k(P) & \geq \left( \frac{\left( 2^d -\frac{\sqrt{d}}{2}\right)^d \pi^{d/2} (2e)^{d/2}}{\sqrt{\frac{4\pi}{d}} \, d^{d/2}} \right)^2 \frac{1}{2^{7d+7}} \\
& \geq 2^{2d^2- d\log d - d \log \left( \frac{128}{\pi e}\right)} .
\end{align*}

Putting the last two inequalities together we get that there exists
a constant $n_0$ depending on $k$, such that for $n>n_0$ we have
$$
t_k(P) \geq  \frac{n^2}{2^{c\sqrt{\log{n}}}}= n^{2-\frac{c}{\sqrt{\log{n}}}},
$$
where $c=2\log(4k+9)$.

To obtain a point set in two dimensions, we project the $d$ dimensional point set to a two dimensional plane in $\R^d$, similarly as in the even case.
\end{proof}

\section{Acknowledgements}
The results presented in this paper are obtained during the authors' participation in 8th
Gremo Workshop on Open Problems -- GWOP 2010. We thank the organizers for inviting us to the
workshop and providing us with a gratifying working environment. Also, we are grateful for the
inspiring conversations with the members of the group of Emo Welzl.

\bibliographystyle{plain}
\bibliography{collinear-paper-01}
\end{document}